\documentclass{article}
\usepackage{amsmath,amsthm,amscd,amssymb}
\usepackage[mathscr]{eucal}
\usepackage{amssymb}
\usepackage{latexsym}
\usepackage{amsmath}

\theoremstyle{definition}
\newtheorem{Def}{Definition}[section]
\newtheorem{Thm}[Def]{Theorem}
\newtheorem{Prop}[Def]{Proposition}
\newtheorem{Rem}[Def]{Remark}

\newtheorem{Lem}[Def]{Lemma}

\numberwithin{equation}{section}

\title{Residue of some Eisenstein series}

\author{Shoyu Nagaoka}

\date{}
\begin{document}

\maketitle

\begin{abstract}
The real analytic Eisenstein series is a special function
that has been studied classically. Its generalization to the case of many variables
has been studied extensively. Moreover, the analytic properties of
certain Eisenstein series on the Siegel modular groups
have also been investigated. The purpose of this study is to
provide concrete forms of the residue of
$E_0^{(m)}(z,s)$ at $s=m/2$.

\end{abstract}
\section{Introduction}
\label{intro}
The real analytic Eisenstein series is a special function
that has been studied classically. It is used in the representation theory of $SL(2,\mathbb{R})$,
and in analytic number theory (e.g., cf.\,\cite{Kub}).
Its generalization to the case of many variables was initiated by
Siegel and later studied more extensively by Langlands \cite{Lang} and Shimura \cite{S2}.

Let
$$
E_k^{(m)}(z,s)=\text{det}(y)^s
                         \sum_{\{ c,d\}}\text{det}(cz+d)^{-k}\,|\text{det}(cz+d)|^{-2s}
$$
be the Eisenstein series of degree $m$ (for a precise definition, see $\S$ \ref{sec3-1}).

Shimura \cite{S2} studied the analytic properties of the Eisenstein series, including
this type. He reveals the holomorphy of $E_k^{(m)}(z,s)$ in $s$ at $s=0$ by
analyzing the Fourier coefficients. The Fourier coefficient essentially consists of
two parts. One is the confluent hypergeometric function, and the other is the Siegl series.
Therefore, the analytic properties of Fourier coefficients, and the Eisenstein
series results in the study of the analytic properties of these two parts. In \cite{S1},
Shimura established the analytic theory of confluent hypergeometric functions
on tube domains and then applied them to analysis of Eisenstein series.
The results of holomorphy of $E_k^{(m)}(z,s)$ studied and extended by Weissauer
\cite{W}.

In Shimura's paper \cite{S2}, apart from the holomorphy, the residue of Eisenstein
series is mentioned. His statement is as follows:
\vspace{2mm}
\\
The residue of the Eisenstein series $E_{(m-1)/2}^{(m)}(z,s)$ at $s=1$ can be expressed 
as the product of $\pi^{-m}$ and a holomorphic modular form of weight $(m-1)/2$,
with rational Fourier coefficients.
\vspace{2mm}
\\
(It is known that the holomorphic modular form stated above is
a (rational) constant multiple of Eisenstein series $E_{(m-1)/2}^{(m)}(z,0)$.)
\vspace{2mm}
\\
Other than his work, few papers mention concrete forms of the residue for the Eisenstein series,
except for the classical work by Kaufhold \cite{Kau} (see $\S$ \ref{other}).

This study aims to provide concrete forms of residue $E_0^{(m)}(z,s)$
at $s=m/2$. Our results strongly depend on Mizumoto's work \cite{Mi}, especially his work on
the Fourier expansion of $E_k^{(m)}(z,s)$, which is a refinement of Maass's result.
\vspace{4mm}
\\
\textbf{Theorem} 
\begin{align*}
& \underset{s=m/2}{{\rm Res}}E_0^{(m)}(z,s)\\
&=\mathbb{A}^{(m)}(y)+\mathbb{B}^{(m)}(y)
                \sum_{h\in\Lambda_m^{(1)}}\sigma_0({\rm cont}(h))\eta_m(2y,\pi h;m/2,m/2)
                \boldsymbol{e}(\sigma(hx)),
\end{align*}
{\it where}
\begin{align*}
\mathbb{A}^{(m)}(y) &=\frac{1}{2}\,\alpha_m(y,m/2)\cdot C_{m-1}^{(m)}(y)
                                                    +\frac{1}{8}\,v(m-1)\text{det}(2y)^{-(m-1)/2}\alpha'_m(y,m/2)\\
                         &+\frac{1}{8}\,\beta'_m(y,m/2),\\
\mathbb{B}^{(m)}(y) &=2^{m-2}\pi^{m\kappa(m)}\text{det}(y)^{m/2}\Gamma_m(m/2)^{-1}\zeta(m)^{-1}\\
                          &\cdot \prod_{j=1}^{m-2}\zeta(m-j)\,\prod_{j=1}^{m-1}\zeta(2m-2j)^{-1}.
 \end{align*}     
\vspace{2mm}
\\
Here, $C_{m-1}^{(m)}$ is the constant term of the completed Koecher--Maass zeta function
$\xi_{m-1}^{(m)}(2y,s)$ at $s=m/2$ (see (\ref{DefC})), and $\alpha_m(y,s)$ and $\beta_m(y,s)$
are defined in (\ref{alpham}) and (\ref{betam}), respectively,  which are essentially products of the gamma functions
and zeta functions.
\vspace{2mm}
\\
In the degree 2 case, the constant $C_{1}^{(2)}(y)$ can be calculated explicitly from the first
Kronecker limit formula.
\vspace{2mm}
\\
\textbf{Corollary}
\\
\begin{align}
\underset{s=1}{{\rm Res}}\,E_0^{(2)}(z,s) 
& =\frac{18}{\pi^2\sqrt{\text{det}(y)}}\left(\frac{1}{2}\gamma+\frac{1}{2}\log\frac{v'}{4\pi}-\log|\eta(W_g)|^2 \right)
\nonumber \\
& \quad +\frac{36\,\text{det}(y)}{\pi^2}\sum_{h\in\Lambda_2^{(1)}}\sigma_0(\text{cont}(h))\eta_2(2y,\pi h;1,1)
             \boldsymbol{e}(\sigma(hx)). \label{A}
\end{align}
(for the notation, see $\S$ \ref{Degree2}.)
In \cite{Na}, the author provided a formula for $E_2^{(2)}(z,0)$ (Siegel Eisenstein series of degrees 2 and 2):
\begin{align}
E_2^{(2)}(z,0)=& 1-\frac{18}{\pi^2\sqrt{\text{det}(y)}}\left(1+\frac{1}{2}\gamma+\frac{1}{2}\log\frac{v'}{4\pi}-\log|\eta(W_g)|^2 \right) \nonumber \\
         &-\frac{72}{\pi^3}\sum_{\substack{0\ne h\in\Lambda_2\\ \text{discr}(h)=\square}}\varepsilon_h
           \sigma_0(\text{cont}(h))\eta_2(2y,\pi h;2,0)\boldsymbol{e}(\sigma(hx))\nonumber \\
         &+288\sum_{0\ne h\in \Lambda_2}\sum_{d\mid\text{cont}(h)}d\,H\left(\frac{|\text{discr}(h)|}{d^2}\right)
              \boldsymbol{e}(\sigma(hz)). \label{B}
\end{align}
It is interesting that the same term appears in each Fourier coefficient in (\ref{A}) and
(\ref{B}).

\section{Notation}
\label{notation}
$1^\circ$\quad
If $a$ is an $m\times m$ matrix, we write it as $a^{(m,n)}$, and as $a^{(m)}$ if $m=n$,
${}^ta$ denotes the transpose of $a$, and $a_{ij}$ denotes the $(i,j)$-entry of $a$. For a
matrix $a$, we write $\sigma (a)$ as the trace of $a$. If the right-hand side is defined as
the identity matrix (resp. zero matrix) of size $m$ and is denoted by $1_m$ (resp. $0_m$).
For a commutative ring $R$ with $1$, we denote $R^{(m,n)}$ by the $R$-module of
all $m\times n$ matrices with entries $R$. We set $R^{(m)}:=R^{(m,m)}$ and $R^m:=R^{(1,m)}$.
\vspace{2mm}
\\
$2^\circ$\quad
We put
\vspace{1mm}
\\
$\bullet$\quad $\displaystyle\kappa (\nu)=\frac{\nu+1}{2}$ for $\nu\in\mathbb{Z}_{\geqq 0}$.
\vspace{2mm}
\\
$\bullet$\quad $\boldsymbol{e}(z)=\text{exp}(2\pi iz)$ for $z\in\mathbb{C}$.
\vspace{2mm}
\\
$\bullet$\quad $H_m:=\{\,z\in\mathbb{C}^{(m)}\,\mid\, {}^tz=z,\;\text{Im}(z)>0\,\}$\,:\;upper
half space.
\vspace{2mm}
\\
$\bullet$\quad $V_m=\{\,x\in \mathbb{R}^{(m)}\,\mid\, {}^tx=x\,\}$.
\vspace{2mm}
\\
$\bullet$\quad $V_m(\mathbb{C})=V_m\otimes_{\mathbb{R}}\mathbb{C}$.
\vspace{2mm}
\\
$\bullet$\quad $P_m:=\{\, x\in V_m\,\mid\,x>0\,\}$.
\vspace{2mm}
\\
$\bullet$\quad $V_m(p,q,r)$: subset of $V_m$ consisting of the elements with $p$ positive,\\
             \qquad\qquad\qquad\quad $q$ negative, $r$ zero eigenvalues.
\vspace{2mm}
\\
$3^\circ$\quad
The function $\Gamma_m(s)$ is defined by
$$
\Gamma_m(s)=\pi^{\frac{m(m-1)}{4}}\prod_{\nu=0}^{m-1}\Gamma\left(s-\frac{\nu}{2}\right)
$$
for $m>0$, and $\Gamma_0(s):=1$. 
\vspace{2mm}
\\
$4^\circ$\quad The set of symmetric half-integral matrices of size $m$ is denoted by
$\Lambda_m$. We place
$$
\Lambda_m^{(\nu)}:=\{\,h\in\Lambda_m\,\mid\,\text{rank}(h)=\nu\,\}.
$$
For $\nu\in\mathbb{Z}$ with $1\leqq \nu\leqq m$,
$$
\mathbb{Z}_{\text{prim}}^{(m,\nu)}=\{\, a\in \mathbb{Z}^{(m,\nu)}\,\mid\,
                                             a\; {\rm is\; primitive} \,\}.
$$
$5^\circ$\quad Throughout the paper, we understand that the product (resp. sum)
over an empty set is equal to $1$ (resp. $0$).
\section{Preliminary}
\label{Pre}
\subsection{Eisenstein series}
\label{sec3-1}
For $m\in\mathbb{Z}_{>0}$ and $k\in 2\mathbb{Z}_{\geq 0}$, let
\begin{equation}
\label{DefEis}
E_k^{(m)}(z,s)=\text{det}(y)^s
                         \sum_{\{ c,d\}}\text{det}(cz+d)^{-k}\,|\text{det}(cz+d)|^{-2s}
\end{equation}
be the Eisenstein series for $\Gamma_m=Sp_m(\mathbb{Z})$ (Siegel modular
groups of degrees $m$). Here, $z=x+iy$ is a variable on $H_m$, $s$ is a complex
variable, and $\{ c,d\}$ runs over a complete system of representatives $\binom{\,*\,\;*\,}{c\;d}$
of $\left\{ \binom{\,*\,*\,}{\,0\;*\,}\in \Gamma_m \right\}\backslash \Gamma_m$.
The right-hand side of (\ref{DefEis}) converges absolutely, locally, and uniformly on the
$$
\{\; (z,s)\in H_m\times\mathbb{C}\;\mid\; \text{Re}(s)>(m+1-k)/2\;\}.
$$
As is well known, the Eisenstein series $E_k^{(m)}(Z,s)$ has a meromorphic continuation
to the whole $s$-plane (Langlands \cite{Lang}, Mizumoto \cite{Mi}).
\subsection{Confluent hypergeometric functions}
\label{sec3-2}
Shimura studied the confluent hypergeometric functions on the tube domains (\cite{S1})
and applied his results to develop the theory of the Eisenstein series (\cite{S2}). In this section,
we summarize some results on the confluent hypergeometric functions that will be used later.

For $g\in P_m$, $h\in V_m$, and $(\alpha,\beta)\in\mathbb{C}^2$,
\begin{equation}
\label{xi}
\xi_m(g,h;\alpha,\beta)=\int_{V_m}\boldsymbol{e}^{-\sigma(hx)}\text{det}(x+ig)^{-\alpha}
                                     \text{det}(x-ig)^{-\beta}dx,
\end{equation}
with $dx=\prod_{i\leqq j}dx_{ij}$, which is convergent for $\text{Re}(\alpha+\beta)>m$;
\begin{equation}
\label{eta}
\eta_m(g,h;\alpha,\beta)= \int_{\substack{V_m\\ x\pm h>0}}
e^{-\sigma(gx)}\text{det}(x+h)^{\alpha-\kappa(m)}
                                     \text{det}(x-h)^{\beta-\kappa(m)}dx,\\
\end{equation}
which is convergent for $\text{Re}(\alpha)>\kappa (m)-1$, $\text{Re}(\beta)>m$.
We also use
\begin{equation*}
\label{etastar}
\eta^*_m(g,h;\alpha,\beta)=\text{det}(g)^{\alpha+\beta-\kappa(m)}\eta_m(g,h;\alpha,\beta),
\end{equation*}
which satisfies the property
\begin{equation*}
\eta_m^*(g[a],h[{}^ta^{-1}];\alpha,\beta)
=\eta^*_m(g,h;\alpha,\beta)
\quad
\text{for all}
\quad a\in GL_m(\mathbb{R}).
\end{equation*}
By \cite{S1}, (1.29),
\begin{equation}
\label{xieta}
\xi_m(g,h;\alpha,\beta)= i^{m(\beta-\alpha)}\cdot 2^m\pi^{m\kappa(m)}
                                          \Gamma_n(\alpha)^{-1}\Gamma_n(\beta)^{-1}
                                        \eta_m(2g,\pi h;\alpha,\beta).
\end{equation}
for $\text{Re}(\alpha)>\kappa(m)-1$, $\text{Re}(\beta)>m$.
When $h=0_m$, the following identity holds:
\begin{Prop}
(Shimura \cite{S1}, (1.31)) {\it If ${\rm Re}(\alpha+\beta)>2\kappa(m)-1$, then}
\begin{align}
\xi_m(g,0_m;\alpha,\beta) &= i^{m\beta-m\alpha}\cdot 2^{m(1-\kappa(m))}(2\pi)^{m\kappa(m)}\nonumber \\
                     & \cdot \Gamma_m(\alpha)^{-1}\Gamma_m(\beta)^{-1}\Gamma_m(\alpha+\beta-\kappa(m)) 
                     \nonumber \\
                     & \cdot \text{det}(2g)^{\kappa(m)-\alpha-\beta}.
                     \label{xi0}
\end{align}
\end{Prop}
For $g\in P_m$, $h\in V_m(p,q,r)$ with $p+q+r=m$, we put
\begin{align*}
& \delta_+(hg):=\text{the product of all positive eigenvalues of} \;\;g^{\frac{1}{2}}hg^{\frac{1}{2}},\\
& \delta_{-}(hg):=\delta_+((-h)g).
\end{align*}
We then put
\begin{align}
\omega_m(g,h;\alpha,\beta) :=&  2^{-p\alpha-q\beta}\Gamma_p\left(\beta-(m-p)/2\right)^{-1}
                                                               \Gamma_q\left(\alpha-(m-q)/2\right)^{-1}\nonumber \\
                                   & \cdot \Gamma_r\left(\alpha+\beta-\kappa(m) \right)^{-1} \nonumber\\
                                   & \cdot \delta_{+}(hg)^{\kappa(m)-\alpha-q/4}
                                     \delta{-}(hg)^{\kappa(m)-\beta-p/4}\,
                                     \eta_n^*(g,h;\alpha,\beta), \label{omegaeta}
\end{align}
One of the main results in \cite{S1} is as follows:
\begin{Thm}
\label{ShimuraMain}
(Shimura \cite{S1}, Theorem 4.2)\;{\it 
Function $\omega_m$ can be continued as a holomorphic function in $(\alpha,\beta)$ to
the whole $\mathbb{C}^2$ and satisfies
$$
\omega_m(g,h;\alpha,\beta)=\omega_m\left(g,h;\kappa(m)+(r/2)-\beta,\kappa(m)+(r/2)-\alpha\right).
$$
}
\end{Thm}
\subsection{Fourier expansion}
\label{FourierEx}
For $m\in\mathbb{Z}_{>0}$ and $k\in 2\mathbb{Z}_{\geqq 0}$, let $s$ be a complex variable,
where $\text{Re}(s)>\kappa (m)$, and let $z=x+iy$ be a variable on $H_m$ with $x\in V_m$.
and $y\in P_m$. Maass (\cite{Ma}) provided a formula for the Fourier expansion of the
Eisenstein series $E_k^{(m)}(z,s)$:
\begin{align}
E_k^{(m)}(z,s) &= \text{det}(y)^s+\text{det}(y)^s\sum_{\nu=1}^m\sum_{h\in\Lambda_\nu}
                                         \sum_{q\in\mathbb{Z}_{\text{prim}}^{(m,\nu)}/GL_\nu(\mathbb{Z})}\nonumber\\
                 & S_\nu(h,2s+k)\xi_{\nu}(y[q],h;s+k,s)\boldsymbol{e}(\sigma(h[{}^tq]x)),  \label{ExMaass}                 
\end{align}
where
\begin{equation}
S_\nu(h,s)=\sum_{r\in V_\nu\cap\, \mathbb{Q}^\nu \text{mod}\, 1}n(r)^{-s}\boldsymbol{e}(\sigma (hr))
\end{equation}
is the singular series (Siegel series), where $n(r)$ is the product of the reduced positive denominators
of the elementary divisors of $r$, and $\xi_\nu$ is the confluent hypergeometric function defined in  (\ref{xi}).

From \cite{Mi}, Lemma 1.1, we have
\begin{Lem}
{\it For $\nu\in\mathbb{Z}_{>0}$, each $h\in\Lambda_\nu$ of {\rm rank} $\lambda>0$ (that is,
$h\in \Lambda_\nu^{(\lambda)}$) is expressed uniquely as
$$
h=h_0[{}^tw]
$$
with $h_0\in\Lambda_\lambda^{(\lambda)}$ and 
$w\in \mathbb{Z}_{{\rm prim}}^{(\nu,\lambda)}/GL_\lambda(\mathbb{Z})$.}
\end{Lem}
Mizumoto provided a reduced formula for $\xi_\nu$ (\cite{Mi}, Lemma 1.4):
\begin{Prop}
\label{xi-eta}
{\it Let $h=h_0[{}^tw]$ be, as in the above lemma. Suppose that ${\rm Re}(s)>\nu$.
Then, in {\rm (\ref{ExMaass})}, we have}
\begin{align}
& \xi_{\nu}(y[q],h;s+k,s) \nonumber \\
& = (-1)^{k\nu/2}2^\nu\pi^{\nu\kappa(\nu)+\lambda(\nu-\lambda)/2)}
                                \cdot \Gamma_{\nu-\lambda}(2s+k-\kappa(\nu))
                                \Gamma_\nu(s)^{-1}\Gamma_\nu(s+k)^{-1} \nonumber \\
                &\cdot {\rm det}(2y[q])^{\kappa(\nu)-k-2s}
                \eta_{\lambda}^*(2y[qw],\pi h_0;s+k+(\lambda-\nu)/2,s+(\lambda-\nu)/2).
\end{align}
\end{Prop}
Let $m,\,\lambda\in\mathbb{Z}$ with $m\geqq \lambda\geqq 1$. We define the subgroup
$\Delta_\lambda^{(m)}$ of $GL_m(\mathbb{Z})$ by
$$
\Delta_\lambda^{(m)}:=\left\{ \begin{pmatrix} * & * \\ 0^{(m-\lambda,\lambda)} & * \end{pmatrix}\in
GL_m(\mathbb{Z}) \;\right\}.
$$
For $r\in \mathbb{Z}_{{\rm prim}}^{(m,\lambda)}$, $u_r$ is an element of $GL_m(\mathbb{Z})$
corresponding to $r$ under a bijection
\begin{align*}
\mathbb{Z}_{{\rm prim}}^{(m,\lambda)}/&GL_\lambda(\mathbb{Z})
 \qquad
\longleftrightarrow
\qquad
GL_m(\mathbb{Z})/\Delta_\lambda^{(m)} \\
& r\qquad\qquad\quad  \longmapsto\qquad\qquad u_r
\end{align*}
which is determined up to the right action of $\Delta_\lambda^{(m)}$.

For $y\in P_m$, we write the Jacobi decomposition of $y[u_r]$ as
\begin{equation*}
y[u_r]={\rm diag}(y[r],g(y,u_r))\begin{bmatrix} 1_\lambda & b \\ 0 & 1_{m-\lambda} \end{bmatrix}.
\end{equation*}
Explicitly, we place $u_r=(r\,r_1)$ and then
\begin{equation}
\label{gyur}
g(y,u_r)=y[r_1]-(y[r])^{-1}[{}^tryr_1].
\end{equation}
\vspace{3mm}
\\
\quad Next, we provide a definition of Koecherer--Maass zeta functions.
For $1\leqq \nu\leqq m$ and $g\in P_m$, we define
\begin{equation}
\label{KM}
\zeta_\nu^{(m)}(g,s):=\sum_{a\in \mathbb{Z}_{{\rm prim}}^{(m,\nu)}/GL_\nu(\mathbb{Z})}
\text{det}(g[a])^{-s}
\end{equation}
which is convergent for $\text{Re}(s)>m/2$. By definition,
$$
\zeta_m^{(m)}(g,s)=\text{det}(g)^{-s}.
$$
For later purposes, we put
$$
\zeta_0^{(m)}(*,s):=1\qquad
\text{for all}
\quad m\in\mathbb{Z}_{\geqq 0}.
$$
Mizumoto's refinement of Maass' expression is as follows:
\begin{Thm} (Mizumoto \cite{Mi}, Theorem 1.8) 
\label{FourierMi}
{\it
For $m\in\mathbb{Z}_{>0}$, $k\in 2\mathbb{Z}_{\geqq 0}$, and ${\rm Re}(s)>m$,
the Eisenstein series $E_k^{(m)}(z,s)$ has the following expression:
\begin{equation}
\label{MiEx}
E_k^{(m)}(z,s)=\sum_{\nu=0}^m\sum_{\lambda=0}^\nu F_{k,\nu,\lambda}^{(m)}(z,s)
\end{equation}
where
\begin{align}
\label{F0}
& F_{k,\nu,0}^{(m)}(z,s) \nonumber \\
& =(-1)^{k\nu/2}2^\nu\pi^{\nu\kappa(\nu)}\Gamma_\nu(2s+k-\kappa(\nu))
            \Gamma_\nu(s)^{-1}\Gamma_\nu(s+k)^{-1} \nonumber \\
& \cdot S_\nu(0_\nu,2s+k)\,{\rm det}(y)^s\,\zeta_\nu^{(m)}(2y,2s+k-\kappa(\nu)),
\end{align}
for $0\leqq \nu\leqq m$, and
\begin{equation}
\label{Fb}
F_{k,\nu,\lambda}^{(m)}(z,s)
=
\sum_{h\in\Lambda_{\lambda}^{(\lambda)}}\sum_{r\in \mathbb{Z}_{{\rm prim}}^{(m,\lambda)}/GL_\lambda(\mathbb{Z})}
b_{k,\nu,\lambda}^{(m)}(h[{}^tr],y,s)\boldsymbol{e}(\sigma(h[{}^tr]x))
\end{equation}
for $1\leqq \lambda\leqq \nu\leqq m$ with
\begin{align}
& b_{k,\nu,\lambda}^{(m)}(h[{}^tr],y,s) \nonumber \\
             &\quad := (-1)^{k\nu/2}2^\nu\pi^{\nu\kappa(\nu)+\lambda(\nu-\lambda)/2}
              \Gamma_{\nu-\lambda}(2s+k-\kappa(\nu))\Gamma_\nu(s)^{-1}\Gamma_\nu(s+k)^{-1}
              \nonumber \\
              &\qquad \cdot S_\nu({\rm diag}(h,0_{\nu-\lambda}),2s+k)){\rm det}(y)^s{\rm det}(2y[r])^{\kappa(\nu)-k-2s}
              \nonumber \\
              &\qquad \cdot \eta_\lambda^*(2y[r],\pi h;s+k+(\lambda-\nu)/2,s+(\lambda-\nu)/2)\nonumber \\
             &\qquad  \cdot \zeta_{\nu-\lambda}^{(m-\lambda)}(2g(y,u_r),2s+k-\kappa(\nu)) \label{Fb1}.
\end{align}
Here, $\zeta_\nu^{(m)}(g,s)$ for $0\leqq \nu\leqq m$ is the Koecher--Maass zeta function
defined in {\rm (\ref{KM})}, and $g(y,u_r)$ is defined by {\rm (\ref{gyur})}. Matrix $h[{}^tr]$ runs
over the set $\Lambda_m^{(\lambda)}$ exactly once if $h$ runs over $\Lambda_\lambda^{(\lambda)}$
and $r$ runs over a complete set of representatives of 
$\mathbb{Z}_{{\rm prim}}^{(m,\lambda)}/GL_\lambda(\mathbb{Z})$.}
\end{Thm}
\subsection{Siegel series}
\label{sec3-4}
In this section, we summarize the results of the Siegel series $S_\nu(h,s)$ that appear in
the Fourier expansions (\ref{ExMaass}) and (\ref{Fb1}).

For $h\in\Lambda_\lambda^{(\lambda)}$, we set
$$
d(h):=(-1)^{[\lambda/2]}2^{-\delta((\lambda-1)/2)}\text{det}(2h)
$$
where
$$
\delta (x):=\begin{cases} 1  &  x\in\mathbb{Z},\\
                                0  &   x\notin \mathbb{Z}
                                \end{cases}
$$
for $x\in\mathbb{Q}$. By \cite{Mi}, (5.1), 
\begin{align}
\label{Siegelreduce}
S_\nu(\text{diag}(h,0_{\nu-\lambda}),s) &=\zeta(s+\lambda-\nu)\zeta(s)^{-1}\nonumber\\
                                                  & \cdot \prod_{j=1}^{\nu-\lambda}(\zeta(2s-\nu-j)\,\zeta(2s-2j)^{-1})
                                                  \nonumber\\
                                                   & \cdot S_\lambda(h,s-\nu+\lambda)
\end{align}
and
\begin{align*}
& S_\lambda(h,s)=\sum_{d\in A(h)}(\text{det}(d))^{\lambda+1-2s}\widehat{S}_{\lambda}(h[d^{-1}],s),\\
& \widehat{S}_{\lambda}(h,s)=\zeta(s)^{-1}\prod_{j=1}^{[\lambda/2]}\zeta(2s-2j)^{-1}
             L\left(s-\lambda/2,\left(\frac{d(h)}{*} \right)\right)^{\delta(\lambda/2)}\prod_p a_p(h,s)
\end{align*}
where $L\left(s,\left(\frac{d(h)}{*}\right)  \right)$ is Dirichlet $L$-function associated to the quadratic
character $\left(\frac{d(h)}{*}\right)$, the product of $p$ runs over the prime divisors of $d(h)$,
$$
A(h):=GL_\lambda(\mathbb{Z})\backslash
     \{\,d\in\mathbb{Z}^{(\lambda)}\,\mid\, \text{det}(d)\ne 0\;\;\text{and}\;\; h[d^{-1}]\in\Lambda_\lambda\,\},
$$
and from \cite{Bo}, we have
\begin{align*}
& a_p(h,s)=\\
&
\begin{cases}
\prod_{j=1}^{r/2}(1-p^{2j-1+\lambda-2s}) & (\lambda,r) \equiv (1,0) \pmod{2},\\
(1+\lambda_p(h)p^{(\lambda+r)/2-s})\prod_{j=1}^{(r-1)/2}(1-p^{2j-1+\lambda-2s}) 
                                                    & (\lambda,r) \equiv (1,1) \pmod{2},\\   
\prod_{j=1}^{(r-1)/2}(1-p^{2j+\lambda-2s})
                                                   & (\lambda,r) \equiv (0,1) \pmod{2},\\ 
(1+\lambda_p(h)p^{(\lambda+r)/2-s})\prod_{j=1}^{r/2-1}(1-p^{2j+\lambda-2s}) 
                                                  & (\lambda,r) \equiv (0,0) \pmod{2}.
\end{cases}
\end{align*}
Here, $r:=r(p)$ is the maximal number, which is the condition 
$h[u] \equiv \begin{pmatrix}h^* & 0 \\ 0 & 0_r \end{pmatrix} \pmod{p}$
for some $u\in\mathbb{Z}^{(\lambda)}$ and $\lambda_p(h):=\left(\frac{d(h^*)}{p} \right)$.
\vspace{2mm}
\\
\begin{Rem}
(1)\quad We understand that $S_0(*,s)=1$.  Therefore, from (\ref{Siegelreduce}), we obtain the following.
Formula for $S_\nu(0_\nu,s)$:
\begin{align}
\label{S0}
S_\nu(0_\nu,s) &=\zeta(s-\nu)\zeta(s)^{-1}\prod_{\nu=1}^\nu(\zeta(2s-\nu-j)\,\zeta(2s-2j)^{-1}) \nonumber\\
               &= \zeta(s-\nu)\zeta(s)^{-1}\prod_{\nu=1}^{[\nu/2]}(\zeta(2s-2\nu-1+2j)\,\zeta(2s-2j)^{-1}).
\end{align}
(2)\quad In the following discussion, the concrete form of $a_p(h,s)$ is not needed, only its
holomorphy in $s$.
\end{Rem}
\subsection{Koecher--Maass zeta function}
\label{sec:3-5}
The Koecher--Maass zeta function $\zeta_\nu^{(m)}(g,s) $ in $ \S$ \ref{FourierEx}.
Analytic properties of this function are important for the analysis of the Fourier coefficient
$F_{k,\nu,\lambda}^{(m)}(z,s)$. In this section, we recall Arakawa's results for the Koecher--Maass
zeta function.

For $1\leqq \nu\leqq m$ and $g\in P_m$, we define the completed Koecher--Maass zeta
function by
\begin{equation}
\label{CompKM}
\xi_{\nu}^{(m)}(g,s):=\prod_{i=0}^{\nu-1}\xi(2s-i)\,\zeta_\nu^{(m)}(g,s)
\end{equation}
where
$$
\xi(s):=\pi^{-s/2}\Gamma(s/2)\,\zeta(s)
$$
and we understand
$$
\xi_0^{(m)}(g,s):=1.
$$
The following result is due to Arakawa, which plays an important role in our investigation.
\begin{Prop}\;(Arakawa \cite{Ara})\quad 
\label{Ara}
{\it (1)\; Suppose $m\geqq 2\nu-1$.
The function $\xi_{\nu}^{(m)}(g,s)$ has simple poles at $s=0,\,\frac{1}{2},\,\cdots\,,\frac{\nu-1}{2}$
and $s=\frac{m-\nu+1}{2},\,\cdots\,,\frac{m}{2}$. For $0\leqq \mu\leqq\nu-1$, the residues of
$\xi_{\nu}^{(m)}(g,s)$ at $s=\frac{\mu}{2}$ and $s=\frac{m-\mu}{2}$ are given by
\begin{align}
& 
\underset{s=\mu/2}{\rm Res}\xi_{\nu}^{(m)}(g,s)=-\frac{1}{2}v(\nu-\mu)\xi_{\mu}^{(m)}(g,\tfrac{\nu}{2}),
\label{Res1} \\
&
\underset{s=(m-\mu)/2}{\rm Res}\xi_{\nu}^{(m)}(g,s)=\frac{1}{2}v(\nu-\mu)
                                               {\rm det}(g)^{-\frac{\nu}{2}}\xi_{\mu}^{(m)}(g^{-1},\tfrac{\nu}{2}),
\label{Res2}
\end{align}
where
$$
v(\nu)=
\begin{cases}
\displaystyle \prod_{i=2}^\nu \xi(i)   &  (\nu\geqq 2),\\
1                                             & (\nu=1).
\end{cases}
$$
(2)\; Suppose $\nu\leqq m\leqq 2\nu-2$. The function 
$\xi_{\nu}^{(m)}(g,s)$ has poles at $s=0,\,\frac{1}{2},\,\cdots\,,\frac{m}{2}$ of which
$s=0,\,\frac{1}{2},\,\cdots\,,\frac{m-\nu}{2}$ and $s=\frac{\nu}{2},\,\frac{\nu+1}{2},\,\cdots\,,\frac{m}{2}$
are simples poles. The poles at $s=\frac{m-\nu+1}{2},\,\frac{m-\nu+2}{2},\,\cdots\,,\frac{\nu-1}{2}$ are
double poles. For $0\leqq \mu\leqq m-\nu$, the residues of $\xi_{\nu}^{(m)}(g,s)$ at $s=\frac{\mu}{2}$
and $s=\frac{m-\mu}{2}$ are given by {\rm (\ref{Res1}) and (\ref{Res2})}, respectively.
}
\end{Prop}
\begin{Rem}
When $m=2$ and $\nu=1$, the function $\zeta_1^{(2)}(g,s)$ appears as a simple
factor of Epstein's zeta function for $g$. Therefore, the residue and constant term at
$s=1$ is explicitly expressed by the Kronecker limit formula (see $\S$ \ref{Degree2}).
\end{Rem}
\section{Residue of Eisenstein series}
\label{Main}
In the rest of this paper, we assume that $m\geqq 2$.
In this section, we provide an explicit formula for
$$
\underset{s=m/2}{\text{Res}}E_0^{(m)}(z,s)
$$
which is the main result of this paper.

\subsection{Fourier coefficient of $\boldsymbol{E_0^{(m)}(z,s)}$}
\label{sec:4-1}
We recall the Fourier expansion
$$
E_0^{(m)}(z,s)=\sum_{\nu=0}^m\sum_{\lambda=0}^\nu F_{0,\nu,\lambda}^{(m)}(z,s)
$$
in Theorem \ref{FourierMi} and study the analytic property, particularly the singularity
of $F_{0,\nu,\lambda}^{(m)}(z,s)$ and $b_{0,\nu,\lambda}^{(m)}(*,y,s)$.
For this purpose, we use the results introduced in $\S$ \ref{Pre} and consider
them dividing into several cases.
\vspace{2mm}
\\
$1^\circ$\quad $(\nu,\lambda)=(0,0)$:
$$
F_{0,0,0}^{(m)}(z,s)=\text{det}(y)^s.
$$
$2^\circ$\quad $(\nu,\lambda)=(\nu,0)$,\;$(0<\nu<m)$:
\begin{align*}
F_{0,\nu,0}^{(m)}(z,s) &= c_{\nu,0}(s)\Gamma_\nu(s)^{-2}\zeta(2s-\nu)\zeta(2s)^{-1}\\
                           &\quad\cdot \prod_{j=1}^\nu\zeta(4s-2j)^{-1}\,\xi_\nu^{(m)}(2y,2s-\kappa(\nu)).
\end{align*}
$3^\circ$\quad $(\nu,\lambda)=(m,0)$:
\begin{align*}
F_{0,m,0}^{(m)}(z,s) &= c_{m,0}(s)\Gamma_m(2s-\kappa(m))\Gamma_m(s)^{-2}
\zeta(2s-m)\zeta(2s)^{-1}\\
                           &\quad\cdot \prod_{j=1}^m(\zeta(4s-m-2j)\zeta(4s-2j)^{-1}).
\end{align*}
$4^\circ$\quad $(\nu,\lambda)=(\nu,\nu)$,\;$(0<\nu\leqq m)$:
\begin{align*}
b_{0,\nu,\nu}^{(m)}(*,y,s) &= c_{\nu,\nu}(s)\Gamma_\nu(s)^{-2}\zeta(2s)^{-1}
                                    \prod_{j=1}^{[\nu/2]}\zeta(4s-2j)^{-1}
                                    \cdot \eta_\nu(*,*;s,s).
\end{align*}
$5^\circ$\quad $(\nu,\lambda)$,\;$(0<\lambda<\nu<m)$:
\begin{align*}
b_{0,\nu,\lambda}^{(m)}(*,y,s) &= c_{\nu,\lambda}(s)\Gamma_\nu(s)^{-2}\zeta(2s)^{-1}
                                           \prod_{j=1}^{\nu-\lambda}\zeta(4s-2j)^{-1}\\
                                 & \cdot \prod_{j=1}^{[\lambda/2]}\zeta(4s-2\nu+2\lambda-2j)^{-1}\\
                                   &\cdot  \eta_\lambda(*,*;s+(\lambda-\nu)/2,s+(\lambda-\nu)/2)
                                    \cdot \xi_{\nu-\lambda}^{(m-\lambda)}(*,2s-\kappa(\nu)).
\end{align*}
$6^\circ$\quad $(\nu,\lambda)=(m,\lambda)$,\;$(0<\lambda<m)$:
\begin{align*}
b_{0,m,\lambda}^{(m)}(*,y,s) &= c_{m,\lambda}(s)\Gamma_{m-\lambda}(2s-\kappa(m))\Gamma_\nu(s)^{-2}\\
                 &\cdot \prod_{j=1}^{m-\lambda}\zeta(4s-m-j)\zeta(4s-2j)^{-1}
                    \prod_{j=1}^{[\lambda/2]}\zeta(4s-2m+2\lambda-2j)^{-1}\\
                 &\cdot \eta_{\lambda}(*,*;s+(\lambda-m)/2.s+(\lambda-m)/2).
\end{align*}
where $c_{\nu,\lambda}(s)$ are holomorphic function in $s$.
\subsection{Analytic property of Fourier coefficients}
\label{sec:4-2}
We investigate the analytic property of $F_{0,\nu,\lambda}^{(m)}(z,s)$ and $b_{0,\nu,\lambda}^{(m)}(*,y,s)$
at $s=m/2$, based on the description in $\S$ \ref{sec:4-1}.
\begin{Prop}
\label{Triang}
{\it Functions $F_{0,\nu,\lambda}^{(m)}(z,s)$ and $b_{0,\nu,\lambda}^{(m)}(*,y,s)$ are holomorphic in $s$ at
$s=m/2$, except for the following three cases:
$$
{\rm (i)}\; \nu=m-1,\;\lambda=0\qquad
{\rm (ii)}\; \nu=m,\;\lambda=0\qquad
{\rm (iii)}\; \nu=m, \;\lambda=m-1.
$$
}
\end{Prop}
\begin{proof}
We use the expressions $1^\circ-6^\circ$ given in the previous section.

The holomorphy for the case $1^\circ$ is
trivial.

First, we consider the case $5^\circ$. The $\Gamma$-factor $\Gamma_\nu(s)^{-2}$,
and $\zeta$-factors $\zeta(2s)^{-1},\cdots$ are all holomorphic at $s=m/2$.
From Theorem \ref{ShimuraMain}, the holomorphy of 
$\eta_\lambda(*,*;s+(\lambda-\nu)/2,s+(\lambda-\nu)/2)$ is reduced to that of
$$
\Gamma_p(s+(\lambda-\nu)/2-q/2)\quad
\Gamma_q(s+(\lambda-\nu)/2-p/2)\qquad (p+q=\lambda),
$$
and are both holomorphic at $s=m/2$. (Note that the factor $\Gamma_r(*)^{-1}$
does not appear in $\eta_\lambda$.)

Function $\xi_{\nu-\lambda}^{(m-\lambda)}(*,2s-\kappa(\nu))$ is holomorphic at $s=m/2$
because $\text{Re}(2s-\kappa(\nu))>(m-\lambda)/2$.
Consequently, the functions in the case $5^\circ$ are holomorphic at $s=m/2$.

By a similar argument, we observe that the functions in the case of $4^\circ$ are holomorphic at $s=m/2$.
The cases we must consider are the cases of $2^\circ$ and $6^\circ$.

In the case of $2^\circ$, only $F_{0,m-1,0}^{(m)}(z,s)$ and $F_{0,m,0}^{(m)}(z,s)$ are non-holomorphic
at $s=m/2$ because $F_{0,m-1,0}^{(m)}(z,s)$ has a factor $\zeta(2s-m+1)\xi_{m-1}^{(m)}(*,2s-m/2)$,
and $F_{0,m,0}^{(m)}(z,s)$ have the factors $\Gamma(2s-m)\zeta(4s-2m+1)$, respectively. In fact,
the factors above have double poles at $s=m/2$.

In the case $6^\circ$, only the function $b_{0,m,m-1}^{(m)}(*,y,s)$ is nonholomorphic at $s=m/2$,
because it contains the factor $\zeta(4s-2m+1)$, which has a simple pole at $s=m/2$.

These facts complete the proof.
\end{proof}
\begin{Rem}
The explicit formulas for $F_{0,m-1,0}^{(m)}(z,s)$, $F_{0,m,0}^{(m)}(z,s)$, \\
and $F_{0,m,1}^{(m)}(z,s)$ will be given in the next sections.
\end{Rem}
Here, we arrange functions $F_{0,\nu,\lambda}^{(m)}$ as follows:

{\small
\begin{center}
\begin{tabular}{llllll}
$F_{0,0,0}^{(m)}$  &    {}                    &     {}         &    {}    &     {}    &    {}       
\vspace{2mm}
\\
$F_{0,1,0}^{(m)}$  & $F_{0,1,1}^{(m)}$ &    {}          &   {}     &     {}    &    {}      
\vspace{2mm}
\\
$\vdots$          &  $\vdots$         &    $\ddots$ &  {}    &     {}     &   {},        
\vspace{2mm}
\\
$F_{0,m-2,0}^{(m)}$ &$F_{0,m-2,1}^{(m)}$&  $\ldots$ & $\ddots$   &    {}     &    {}       
\vspace{2mm}
\\
$\boldsymbol{F_{0,m-1,0}^{(m)}}$ &$F_{0,m-1,1}^{(m)}$&$F_{0,m-1,2}^{(m)}$&$ \ldots$&$F_{0,m-1,m-1}^{(m)}$ &{} 
\vspace{2mm}
\\
$\boldsymbol{F_{0,m,0}^{(m)}}$ &$\boldsymbol{F_{0,m,1}^{(m)}}$&$F_{0,m,2}^{(m)}$& $\ldots$&$F_{0,m,m-1}^{(m)}$ &$F_{0,m,m}^{(m)}$

\end{tabular}
\end{center}
}
The proposition asserts that only functions $F_{0,\nu,\lambda}^{(m)}$ printed in bold are non-holomorphic
at $s=m/2$.
\subsection{Residue of the constant term}
\label{sec:4-3}
We investigate the analytic property of the constant term
$$
\sum_{\nu=0}^mF_{0,\nu,0}^{(m)}(z,s)
$$
at $s=m/2$. More specifically, we show that the constant term has a simple pole at $s=m/2$
and calculate the residue.

By Proposition \ref{Triang}, it is sufficient to investigate only $F_{0,m-1,0}^{(m)}(z,s)$ and\\
$F_{0,m,0}^{(m)}(z,s)$ as far as considering the residue.
\vspace{2mm}
\\
\textbf{Analysis of} $\boldsymbol{F_{0,m-1,0}^{(m)}(z,s)}:$
\vspace{2mm}
\\
From the definition of $F_{0,\nu,0}^{(m)}(z,s)$ (see (\ref{F0})), we have
\begin{align}
\label{Fm-1}
 F_{0,m-1,0}^{(m)}(z,s) 
 & =2^{m-1}\pi^{2(m-1)s}\text{det}(y)^s\Gamma_{m-1}(s)^{-2}\zeta(2s-m+1)\zeta(2s)^{-2}
         \nonumber \\
&\quad \cdot\prod_{j=1}^{m-1}\zeta(4s-2j)^{-1}\cdot\xi_{m-1}^{(m)}(2y,2s-m/2). 
\end{align}
(We rewrote (\ref{F0}) with the complete Koecher--Maass zeta function $\xi_{m-1}^{(m)}$.)

We separate $F_{0,m-1,0}^{(m)}(z,s)$ into holomorphic and non-holomorphic parts.
We define the function $\alpha_m(y,s)$ by
$$
F_{0,m-1,0}^{(m)}(z,s)=\zeta(2s-m+1)\xi_{m-1}^{(m)}(2y,2s-m/2)\cdot \alpha_m(y,s).
$$
Explicitly,
\begin{equation}
\label{alpham}
\alpha_m(y,s):=2^{m-1}\pi^{2(m-1)s}\text{det}(y)^s\Gamma_{m-1}(s)^{-2}\zeta(2s)^{-1}
                  \prod_{j=1}^{m-1}\zeta(4s-2j)^{-1}.
\end{equation}
Functions $\zeta(2s-m+1)$ and $\xi_{m-1}^{(m)}(2y,2s-m/2)$ have a simple pole at
$s=m/2$ (for $\xi_{m-1}^{(m)}$, see Proposition \ref{Ara}), and $\alpha_m(y,s)$
is holomorphic at $s=m/2$. These facts imply that $F_{0,m-1,0}^{(m)}(z,s)$ has a double pole at
$s=m/2$.

We set
$$
F_{0,m-1,0}^{(m)}(z,s)=\sum_{l=-2}^\infty A_l^{(m)}(y)(s-m/2)^l
\quad (\text{Laurent expansion at}\; s=m/2)
$$
and calculate $A_{-2}^{(m)}(y)$ and $A_{-1}^{(m)}(y)$.

As a preparation, we investigate the analytic behavior of $\xi_{m-1}^{(m)}(2y,2s-m/2)$
at $s=m/2$. We consider the completed Koecher--Maass zeta function $\xi_{m-1}^{(m)}(2y,s)$.
According to Arakawa's result (Proposition \ref{Ara}), this function has a simple pole with
residue
\begin{equation}
\label{Resm-1}
\underset{s=m/2}{\text{Res}}\xi_{m-1}^{(m)}(2y,s)=\frac{1}{2}v(m-1)\text{det}(2y)^{-(m-1)/2}
\end{equation}
\begin{Def}
\label{DefC}
Define a constant $C_{m-1}^{(m)}(y)$ by
$$
C_{m-1}^{(m)}(y):=
\lim_{s\to m/2}\left(\xi_{m-1}^{(m)}(2y,s)-\underset{s=m/2}{\text{Res}}\xi_{m-1}^{(m)}(2y,s)(s-m/2)^{-1}\right).
$$
\end{Def}
That is, $C_{m-1}^{(m)}(y)$ is the constant term of the Laurent expansion of $\xi_{m-1}^{(m)}(2y,s)$
at $s=m/2$.
\begin{Rem}
(1)\; It should be noted that the constant $C_{m-1}^{(m)}(y)$ is defined from $\xi_{m-1}^{(m)}(2y,s)$
not $\xi_{m-1}^{(m)}(y,s)$, and the constant term of $\xi_{m-1}^{(m)}(2y,2s-m/2)$ at $s=m/2$
is equal to that of $\xi_{m-1}^{(m)}(2y,s)$.
\\
(2)\; In the case $m=2$, the constant $C_1^{(2)}(y)$ is explicitly expressed by the Kronecker
limit formula (see $\S$ \ref{Degree2}).
\end{Rem}
\begin{Prop} 
{\it Explicit forms of $A_{-2}(y)$ and $A_{-1}(y)$ are given as follows:
\begin{align}
A_{-2}^{(m)}(y) &= \frac{1}{8}\,v(m-1){\rm det}(2y)^{-(m-1)/2}\alpha_m(y,m/2),
\label{EA-2}
\\
A_{-1}^{(m)}(y) 
&=\alpha_m(y,m/2)\left(\frac{1}{2}\,C_{m-1}^{(m)}(y)+\frac{\gamma}{4}\,v(m-1){\rm det}(2y)^{-(m-1)/2}\right)
        \nonumber \\
&\qquad +\frac{1}{8}\,v(m-1){\rm det}(2y)^{-(m-1)/2}\cdot \alpha'_m(y,m/2),
\label{EA-1}
\end{align}
where $\gamma$ is the Euler constant and 
$\alpha'_m(y,m/2)=\left.\frac{d}{ds}\alpha_m(y,s)\right|_{s=m/2}$.}
\end{Prop}
\begin{proof}
The formulas are derived from the expression
\begin{align*}
\zeta(2s-m+1)&\xi_{m-1}^{(m)}(2y,2s-m/2) \\
                  &= \frac{1}{8}\,v(m-1)\text{det}(2y)^{-(m-1)/2}(s-m/2)^{-2}\\
             &+\left(\frac{1}{2}\,C_{m-1}^{(m)}(y)+\frac{\gamma}{4}\,v(m-1)\text{det}(2y)^{-(m-1)/2} \right)(s-m/2)^{-1}\\
             &+(\text{a holomorphic function at}\; s=m/2). 
\end{align*}
\end{proof}
\noindent
\textbf{Analysis of} $\boldsymbol{F_{0,m,0}^{(m)}(z,s)}:$
\vspace{2mm}
\\
By definition (\ref{Fb}),
\begin{align}
\label{Fm}
 F_{0,m,0}^{(m)}(z,s) 
 & =2^{-2ms+m(m+3)/2}\pi^{m(m+1)/2}\text{det}(y)^{-s+(m+1)/2}
\Gamma_m(2s-\kappa(m)) \nonumber \\
& \cdot \Gamma_{m}(s)^{-2}\zeta(2s-m)\zeta(2s)^{-1}
         \cdot\prod_{j=1}^{m}(\zeta(4s-m-j)\zeta(4s-2j)^{-1}).
\end{align}
Similar to that in case $F_{0,m,-1,0}^{(m)}(z,s)$, we define the function $\beta_m(y,s)$ as
$\alpha_m(y,s)$:
$$
F_{0,m,0}^{(m)}(z,s)=\Gamma(2s-m)\zeta(4s-2m+1)\,\beta_m(y,s).
$$
Explicitly,
\begin{align}
\label{betam}
\beta_m(y,s):&=2^{-2ms+m(m+3)/2}\pi^{(m^2+2m-1)/2}\text{det}(y)^{-s+(m+1)/2}\nonumber \\
                            & \cdot \Gamma_{m-1}(2s-\kappa(m))\Gamma_{m}(s)^{-2}\zeta(2s)^{-1} \nonumber \\
                            &\cdot \prod_{j=1}^{m-2}\zeta(4s-m-j)\cdot\prod_{j=1}^{m-1}\zeta(4s-2j)^{-1}.
\end{align}
Functions $\Gamma(2s-m)$ and $\zeta(4s-2m+1)$ have a simple pole at $s=m/2$, respectively,
and $\beta_m(y,s)$ is holomorphic at $s=m/2$. Consequently, we observe that $F_{0,m,0}^{(m)}(z,s)$,
has a double pole at $s=m/2$, as in the previous case.

We set
$$
F_{0,m,0}^{(m)}(z,s)=\sum_{l=-2}^\infty B_l^{(m)}(y)(s-m/2)^l
$$
and calculate $B_{-2}^{(m)}(y)$ and $B_{-1}^{(m)}(y)$.
\begin{Prop} 
\label{B-2-1}
{\it The explicit forms of $B_{-2}^{(m)}(y)$ and $B_{-1}^{(m)}(y)$ are as follows:
\begin{align}
& B_{-2}^{(m)}(y) = \frac{1}{8}\,\beta_m(y,m/2),
\label{EB-2}
\\
& B_{-1}^{(m)}(y) 
=\frac{\gamma}{4}\,\beta_m(y,m/2)+\frac{1}{8}\,\beta'_m(y,m/2)
\label{EB-1}
\end{align}
where
$\beta'_m(y,m/2)=\left.\frac{d}{ds}\beta_m(y,s)\right|_{s=m/2}$.}
\end{Prop}
\begin{proof}
Function $\Gamma(2s-m)\zeta(4s-2m+1)$ has the Laurent expansion as
\begin{align*}
 &\Gamma(2s-m)\zeta(4s-2m+1) \\
 & = \frac{1}{8}\,(s-m/2)^{-2}+\frac{\gamma}{4}\,(s-m/2)^{-1}+(\text{a holomorphic function at}\; s=m/2). 
\end{align*}
The formulas for $B_{-2}^{(m)}(y)$ and $B_{-1}^{(m)}(y)$ are obtained from this expression.
\end{proof}
An important point is the following relationship between $A_{-2}^{(m)}(y)$ and\\
 $B_{-2}^{(m)}(y)$.
\begin{Prop}
{\it The following identity holds.}
\begin{equation}
\label{A=-B}
A_{-2}^{(m)}(y)=-B_{-2}^{(m)}(y).
\end{equation}
\end{Prop}
\begin{proof}
A direct calculation shows that
\begin{align}
A_{-2}^{(m)}(y) &= 2^{(-m^2+4m-8)/2}\pi^{(m^3+3m-2)/4}\text{det}(2y)^{1/2}\zeta(m)^{-1} \nonumber \\
        & \cdot \prod_{i=2}^{m-1}\Gamma(i/2)\prod_{j=0}^{m-2}\Gamma((m-j)/2)^{-2}\prod_{i=2}^{m-1}\zeta(i)
        \prod_{j=1}^{m-1}\zeta(2m-2j)^{-1}.\label{ExplicitA-2}
\end{align}
Meanwhile,
\begin{align}
B_{-2}^{(m)}(y) &= -2^{(-m^2+4m-8)/2}\pi^{(m^3+3m)/4}\text{det}(2y)^{1/2}\zeta(m)^{-1} \nonumber \\
        & \cdot \prod_{j=0}^{m-2}\Gamma((m-1-j)/2)\prod_{j=0}^{m-1}\Gamma((m-j)/2)^{-2}\nonumber \\
        & \cdot  \prod_{j=1}^{m-2}\zeta(m-j)  \prod_{j=1}^{m-1}\zeta(2m-2j)^{-1}.\label{ExplicitB-2}
\end{align}
Noting that
$$
\prod_{j=0}^{m-2}\Gamma((m-1-j)/2)=\Gamma(1/2)\cdot\prod_{i=2}^{m-1}\Gamma(i/2),
$$
we conclude that $A_{-2}^{(m)}(y)=-B_{-2}^{(m)}(y)$.
\end{proof}
From this proposition, we observe that the singularity of function
$$
F_{0,m-1,0}^{(m)}(z,s)+F_{0,m,0}^{(m)}(z,s)
$$
at $s=m/2$ is a simple pole.
\begin{Thm}
{\it The residue of the constant term is as follows: }
\begin{align}
& \underset{s=m/2}{{\rm Res}}\sum_{\nu=0}^{m}F_{0,\nu,0}^{(m)}(z,s)
  =\underset{s=m/2}{{\rm Res}}(F_{0,m-1,0}^{(m)}(z,s)+F_{0,m,0}^{(m)}(z,s)) \nonumber \\
&=\frac{1}{2}\,\alpha_m(y,m/2)\cdot C_{m-1}^{(m)}(y)
    +\frac{1}{8}\,v(m-1)\text{det}(2y)^{-(m-1)/2}\alpha'_m(y,m/2)\nonumber \\
   &\quad +\frac{1}{8\,}\beta'_m(y,m/2).
    \label{ResConst}
\end{align}
\end{Thm}
\begin{proof}
We have
\begin{align*}
& \underset{s=m/2}{{\rm Res}}(F_{0,m-1,0}^{(m)}(z,s)+F_{0,m,0}^{(m)}(z,s))
   =A_{-1}^{(m)}(y)+B_{-1}^{(m)}(y)\\
&=\frac{1}{2}\,\alpha_m(y,m/2)\cdot C_{m-1}^{(m)}(y)+\frac{\gamma}{4}\,v(m-1)\text{det}(2y)^{-(m-1)/2}\alpha_m(y,m/2)\\
& \quad +\frac{1}{8}\,v(m-1)\text{det}(2y)^{-(m-1)/2}\alpha'_m(y,m/2)+\frac{1}{8}\,\beta'_m(y,m/2)+
      \frac{\gamma}{4}\,\beta_m(y,m/2).
\end{align*}
By the identity (\ref{A=-B}), the sum of the second and fifth terms in the last formula is equal to zero.
This implies (\ref{ResConst}).
\end{proof}
\subsection{Calculation of $\boldsymbol{F_{0,m,1}^{(m)}(z,s)}$}
\label{sec:4-4}
We have one more non-holomorphic term, that is, $F_{0,m,1}^{(m)}(z,s)$.
\begin{Prop}
{\it Function $F_{0,m,1}^{(m)}(z,s)$ has the following expression:}
\begin{align}
& F_{0,m,1}^{(m)}(z,s) \nonumber \\
&  =2^m\pi^{m\kappa(m)}\text{det}(y)^s\Gamma_m(s)^{-2}\zeta(2s)^{-1}
      \prod_{j=1}^{m-1}(\zeta(4s-m-j)\cdot\zeta(4s-2j)^{-1}) \nonumber \\
&\cdot \sum_{h\in\Lambda_m^{(1)}}\sigma_{m-2s}(\text{cont}(h))\cdot \eta_m(2y,\pi h;s,s)
          \boldsymbol{e}(\sigma (hx)),
\end{align}
{\it where $\sigma_s(a)=\sum_{0<d|a}d^s$ and for $0\ne h\in\Lambda_m$, 
${\rm cont}(h):={\rm max}\{l\in\mathbb{N}\,\mid\,l^{-1}h\in\Lambda_m\}$.}
\end{Prop}
\begin{proof}
By definition (\ref{Fb1}),
\begin{align}
& F_{0,m,1}^{(m)}(z,s) \nonumber \\
& =2^m\pi^{m\kappa(m)+(m-1)/2}\Gamma_{m-1}(2s-\kappa(m))\Gamma_m(s)^{-2}
\text{det}(y)^s\text{det}(2y)^{\kappa(m)-2s} \nonumber \\
&\quad \cdot\sum_{0\ne h\in\mathbb{Z}}\sum_{w\in\mathbb{Z}_{\text{prim}}^{(m,1)}/\{\pm 1\}}
                   S_m(\text{diag}(h,0_{m-1}),2s) \nonumber \\
& \qquad\qquad \cdot (2y[w])^{2s-m}\eta_1(2y[w],\pi h;s-(m-1)/2,s-(m-1)/2)\boldsymbol{e}(\sigma(h[{}^tw]x).
\label{DefF0m0}
\end{align}
It follows from (\ref{Siegelreduce}) that
\begin{align}
& S_m(\text{diag}(h,0_{m-1}),2s)\nonumber \\
& =\zeta(2s-m+1)\zeta(2s)^{-1}\prod_{j=1}^{m-1}(\zeta(4s-m-j)\,\zeta(4s-2j)^{-1})
        S_1(h,2s-m+1) \nonumber \\
        &=\sigma_{m-2s}(h)\zeta(2s)^{-1}\prod_{j=1}^{m-1}(\zeta(4s-m-j)\,\zeta(4s-2j)^{-1}).
        \label{SiegelS}
\end{align}
In the above, we used the formula $S_1(h,s)=\zeta(s)^{-1}\sigma_{1-s}(h)$ for $h\in\mathbb{Z}_{>0}$.

From Proposition \ref{xi-eta}, function $\eta_1$ is expressed as
\begin{align}
& \eta_1(2y[w],\pi h;s-(m-1)/2,s-(m-1)/2) \nonumber \\
& =\pi^{-(m-1)/2}\Gamma_{m-1}(2s-\kappa(m))^{-1}\text{det}(2y)^{2s-\kappa(m)}(2y[w])^{-2s+m} \nonumber \\
&\quad \cdot \eta_m(2y,\pi h[{}^tw],s,s) . \label{eta1}
\end{align}
Substituting (\ref{SiegelS}) and (\ref{eta1}) into (\ref{DefF0m0}), we obtain the following expression:
\end{proof}
From the above proposition, we obtain the following result:
\begin{Thm}
{\it Function $F_{0,m,1}^{(m)}(z,s)$ has a simple pole at $s=m/2$. }
\begin{align}
 \underset{s=m/2}{{\rm Res}}F_{0,m,1}^{(m)}(z,s)  &=
2^{m-2}\pi^{m\kappa(m)}\text{det}(y)^{m/2}\Gamma_m(m/2)^{-1}\zeta(m)^{-1} \nonumber \\
& \cdot \prod_{j=1}^{m-2}\zeta(m-j)\prod_{j=1}^{m-1}\zeta(2m-2j)^{-1} \nonumber \\
& \cdot \sum_{h\in\Lambda_m^{(1)}}\sigma_0({\rm cont}(h))\eta_m(2y,\pi h;m/2,m/2)
\boldsymbol{e}(\sigma(hx)). \label{Main2}
\end{align}
\end{Thm}
\begin{proof}
In the expression of $F_{0,m,1}^{(m)}(z,s)$, only the last factor $\zeta(4s-2m+1)$ in the product
$\prod_{j=1}^{m-1}\zeta(4s-m-j)$ has a simple pole at $s=m/2$ with residue $1/4$.
From this fact, we obtain (\ref{Main2}).
\end{proof}
\subsection{Conclusion}
\label{sec:4-5}
We summarize our results in the previous sections.

The following is a main result of this study.
\begin{Thm}
\label{Conclusion}
\begin{align*}
& \underset{s=m/2}{{\rm Res}}E_0^{(m)}(z,s)\\
&=\mathbb{A}^{(m)}(y)+\mathbb{B}^{(m)}(y)
                \sum_{h\in\Lambda_m^{(1)}}\sigma_0({\rm cont}(h))\eta_m(2y,\pi h;m/2,m/2)
                \boldsymbol{e}(\sigma(hx)),
\end{align*}
{\it where}
\begin{align*}
\mathbb{A}^{(m)}(y) &=\frac{1}{2}\,\alpha_m(y,m/2)\cdot C_{m-1}^{(m)}(y)
                                                    +\frac{1}{8}\,v(m-1)\text{det}(2y)^{-(m-1)/2}\alpha'_m(y,m/2)\\
                         &+\frac{1}{8}\,\beta'_m(y,m/2),\\
\mathbb{B}^{(m)}(y) &=2^{m-2}\pi^{m\kappa(m)}\text{det}(y)^{m/2}\Gamma_m(m/2)^{-1}\zeta(m)^{-1}\\
                          &\cdot \prod_{j=1}^{m-2}\zeta(m-j)\,\prod_{j=1}^{m-1}\zeta(2m-2j)^{-1}.
 \end{align*}                                                                          
\end{Thm}
\section{Remarks}
\subsection{Low degree cases}
In this section, we provide more explicit formulas for $\text{Res}_{s=m/2}E_0^{(m)}(z,s)$.
$m=2,\,3$. We used the notations in the previous sections as they are.
\subsubsection{Case $\boldsymbol{m=2}$}
\label{Degree2}
In this case, the constant $C_1^{(2)}(y)$ appearing in the term $\mathbb{A}^{(2)}(y)$
can be expressed more explicitly because we can apply the Kronecker limit formula.
For $g\in P_2$, we consider the Epstein zeta function
$$
\zeta_g(s):=\sum_{\boldsymbol{0}\ne a\in\mathbb{Z}^{(2,1)}/\{\pm 1\}}g[a]^{-s},
\qquad\quad \text{Re}(s)>1.
$$
The first Kronecker limit formula asserts that $\zeta_g(s)$ has the following expression:
\begin{align}
& \zeta_g(s)=\frac{1}{2}(4\text{det}(g))^{-s/2}\left[\frac{2\pi}{s-1}+4\pi\beta(g)+O(s-1) \right]
\label{Epstein}\\
& \beta(g)=\gamma+\frac{1}{2}\log\frac{v'}{2\sqrt{\text{det}(g)}}-\log|\eta(W_g)|^2.
\end{align}
Here, for $g=\begin{pmatrix} v' & w \\ w & v \end{pmatrix}\in P_2$,
$$
W_g:=\frac{w+i\sqrt{\text{det}(g)}}{v'}\in H_1,
$$
and $\eta(z)$ is the Dedekind eta function:
$$
\eta(z)=\boldsymbol{e}(z/24)\prod_{n=1}^\infty (1-\boldsymbol{e}^n(z)),\qquad z\in H_1.
$$
The relationship between the complete zeta function $\xi_1^{(2)}(g,s)$ and $\zeta_g(s)$
is expressed as follows:
$$
\xi_1^{(2)}(g,s)=\xi(2s)\zeta_1^{(2)}(g,s)=\pi^{-s}\Gamma(s)\zeta_g(s).
$$
Therefore, the constant $C_1^{(2)}(y)$, which is the constant term of $\xi_1^{(2)}(2y,s)$, 
can be expressed as
$$
C_1^{(2)}(y)=\frac{1}{2}(4\text{det}(y))^{-1/2}\left(\gamma+\log\frac{v'}{8\pi}-\log(\text{det}(y))
                -2\log|\eta(W_g)|^2\right).
$$
Concerning $\alpha_m(y,m/2)$, $\alpha'_m(y,m/2)$, $\cdots$ appearing in $\mathbb{A}^{(m)}(y)$, 
we can calculate them explicitly as
\begin{align*}
\alpha_2(y,1) &=72\pi^{-2}\text{det}(y)\,\zeta(2)^{-2},\\
\alpha'_2(y,1) &=2\pi^2\text{det}(y)\,\zeta(2)^{-2}(2\log\pi+\log(\text{det}(y))+2\gamma-6\zeta'(2)\,\zeta(2)^{-1}),\\
\beta_2(y,1) &=-\pi^2\text{det}(y)^{1/2}\,\zeta(2)^{-2},\\
\beta'_2(y,1) &= 2\pi^2\text{det}^{1/2}\,\zeta(2)^{-2}
                     \left(2\log 2+\tfrac{1}{2}\log(\text{det}(y))-\gamma+2\zeta'(0)+3\zeta'(2)\zeta(2)^{-1}\right).
\end{align*}
From these formula,
\begin{align*}
& \mathbb{A}^{(2)}(y)=\underset{s=1}{{\rm Res}}(F_{0,1,0}^{(2)}(z,s)+F_{0,2,0}^{(2)}(z,s))\\
& =\frac{1}{2}\,\alpha_2(y,1)\cdot C_1^{(2)}(y)+\frac{1}{8}\,v(1)\cdot\text{det}(y)^{-1/2}\alpha'_2(y,1)+\frac{1}{8}\,\beta'_2(y,1)\\
& =\frac{18}{\pi^2\sqrt{\text{det}(y)}}\left(\frac{1}{2}\gamma+\frac{1}{2}\log\frac{v'}{4\pi}-\log|\eta(W_g)|^2 \right).
\end{align*}
Combining with $\text{Res}_{s=1}F_{0,2,1}^{(3)}(z,s)$, we obtain
\begin{Prop}
\begin{align}
 \underset{s=1}{{\rm Res}}\,E_0^{(2)}(z,s) &
=\frac{18}{\pi^2\sqrt{\text{det}(y)}}\left(\frac{1}{2}\gamma+\frac{1}{2}\log\frac{v'}{4\pi}-\log|\eta(W_g)|^2 \right)\nonumber
\\
& \quad +\frac{36\text{det}(y)}{\pi^2}\sum_{h\in\Lambda_2^{(1)}}\sigma_0(\text{cont}(h))\eta_2(2y,\pi h;1,1)
             \boldsymbol{e}(\sigma(hx)). \label{Degrre2Main}
\end{align}
\end{Prop}
\begin{Rem}
In \cite{Na}, the author provided a formula for $E_2^{(2)}(z,0)$ (Siegel Eisenstein series of degree 2 and weight 2):
\begin{align}
E_2^{(2)}(z,0)=& 1-\frac{18}{\pi^2\sqrt{\text{det}(y)}}\left(1+\frac{1}{2}\gamma+\frac{1}{2}\log\frac{v'}{4\pi}-\log|\eta(W_g)|^2 \right) \nonumber\\
         &-\frac{72}{\pi^3}\sum_{\substack{0\ne h\in\Lambda_2\\ \text{discr}(h)=\square}}\varepsilon_h
           \sigma_0(\text{cont}(h))\eta_2(2y,\pi h;2,0)\boldsymbol{e}(\sigma(hx))\nonumber \\
         &+288\sum_{0\ne h\in \Lambda_2}\sum_{d\mid\text{cont}(h)}d\,H\left(\frac{|\text{discr}(h)|}{d^2}\right)
              \boldsymbol{e}(\sigma(hz)).\label{Degree2Weight2}
\end{align}
Here, $H(N)$ is the Kronecker--Hurwitz class number and $\varepsilon_h=1/2$ if $\text{rank}(h)=1;$
$=1$ if $\text{rank}(h)=2$.

It is interesting that the same term appears in each Fourier coefficient in (\ref{Degrre2Main}) and
(\ref{Degree2Weight2}).
\end{Rem}
\subsubsection{Case $\boldsymbol{m=3}$}
\label{Degree3}
From Theorem \ref{Conclusion}, we can write 
\begin{align*}
& \underset{s=3/2}{{\rm Res}}E_0^{(3)}(z,s)\\
&=\mathbb{A}^{(3)}(y)+\mathbb{B}^{(3)}(y)
                \sum_{h\in\Lambda_3^{(1)}}\sigma_0({\rm cont}(h))\eta_3(2y,\pi h;3/2,3/2)
                \boldsymbol{e}(\sigma(hx)).
\end{align*}
The quantities $\mathbb{A}^{(3)}(y)$ and $\mathbb{B}^{(3)}(y)$ are given as follows:
\begin{align*}
& \mathbb{A}^{(3)}(y) \\
& = 2^3\pi^4\text{det}(y)^{3/2}\zeta(2)^{-1}\zeta(3)^{-1}\zeta(4)^{-1}\cdot C_2^{(3)}(y)
                         +2^{-2}\pi^3\text{det}(y)^{1/2}\zeta(3)^{-1}\zeta(4)^{-1}
\\
& \quad\cdot (-2\Gamma'(1)-4\zeta'(2)\zeta(2)^{-1}+4\zeta'(0)+2\log(\text{det}(y))+4\log\pi+6\log 2),
\vspace{2mm}
\\
& \mathbb{B}^{(3)}(y)\\ 
& = 2^2\pi^{7/2}\text{det}(y)^{3/2}\zeta(3)^{-1}\zeta(4)^{-1}.
\end{align*}
\begin{Rem}
In the above formulas, we may substitute
$$
\zeta(2)=\pi^2/6,\quad \zeta(4)=\pi^4/90,\quad \zeta'(0)=(-\log 2\pi)/2,\quad \Gamma'(1)=-\gamma.
$$
\end{Rem}
\subsection{Residue at the other point}
\label{other}
The residue we considered above was to $s=m/2$, and it is represented as a Fourier series.
The case $\text{Res}_{s=(m+1)/2}E_0^{(m)}(z,s)$ is easier than in the above case. In fact,
it becomes a constant, explicitly
\begin{align*}
& \underset{s=(m+1)/2}{{\rm Res}}E_0^{(m)}(z,s)\\
& =\underset{s=(m+1)/2}{{\rm Res}}\xi(2s-m)\xi(2s)^{-1}\prod_{j=1}^{[m/2]}(\xi(4s-2m-1+2j)\,\xi(4s-2j)^{-1}).
\end{align*}
\begin{Rem}
Kaufhold \cite{Kau} noted that the residue of
$$
\varPhi_0(s):=E_0^{(2)}(z,s/2)
$$
at $s=3$ is $90\pi^{-2}$. This is a special case of the above formula because
$$
 \underset{s=3/2}{{\rm Res}}E_0^{(2)}(z,s)= \underset{s=3/2}{{\rm Res}}\xi(2s-2)\xi(2s)^{-1}\xi(4s-3)\xi(4s-2)^{-1}
=\frac{45}{\pi^2}.
$$
\end{Rem}


\begin{flushleft}
Shoyu Nagaoka\\
Department of Mathematics\\
Yamato University\\
Suita\\
Osaka 564-0082\\
Japan
\end{flushleft}

\end{document}